\documentclass[12pt]{amsart}
\usepackage{amscd,amsmath,amsthm,amssymb,amsfonts,enumerate,hyperref}

\def\NZQ{\mathbb}

\def\ZZ{{\NZQ Z}}

\def\Gc{{\mathcal G}}

\def\Gc{{\mathcal G}}

\def\Bc{{\mathcal B}}

\def\opn#1#2{\def#1{\operatorname{#2}}} 
	\opn\chara{char} \opn\length{\ell} \opn\pd{pd} \opn\rk{rk}
	\opn\projdim{proj\,dim} \opn\injdim{inj\,dim} \opn\rank{rank}
	\opn\depth{depth} \opn\grade{grade} \opn\height{height}
	\opn\embdim{emb\,dim} \opn\codim{codim}
	\opn\Cl{Cl}
	
	\opn\Tr{Tr} \opn\bigrank{big\,rank}
	\opn\superheight{superheight}\opn\lcm{lcm}
	\opn\trdeg{tr\,deg}
	\opn\rdeg{rdeg}
	\opn\reg{reg} \opn\lreg{lreg} \opn\ini{in} \opn\lpd{lpd}
	\opn\size{size} \opn\sdepth{sdepth}
	\opn\link{link}\opn\fdepth{fdepth}\opn\lex{lex}
	\opn\tr{tr}
	\opn\type{type}
	\opn\gap{gap}
	\opn\arithdeg{arith-deg}
	\opn\revlex{revlex}
	\opn\div{div} \opn\Div{Div} \opn\cl{cl} \opn\Cl{Cl}
	\opn\Spec{Spec} \opn\Supp{Supp} \opn\supp{supp} \opn\Sing{Sing}
	\opn\Ass{Ass} \opn\Min{Min}\opn\Mon{Mon}
	\opn\Ann{Ann} \opn\Rad{Rad} \opn\Soc{Soc}
	\opn\Im{Im} \opn\Ker{Ker} \opn\Coker{Coker} \opn\Am{Am}
	\opn\Hom{Hom} \opn\Tor{Tor} \opn\Ext{Ext} \opn\End{End}
	\opn\Aut{Aut} \opn\id{id}
	
	\opn\nat{nat}
	\opn\pff{pf}
	\opn\Pf{Pf} \opn\GL{GL} \opn\SL{SL} \opn\mod{mod} \opn\ord{ord}
	\opn\Gin{Gin} \opn\Hilb{Hilb}\opn\sort{sort}
	\opn\PF{PF}\opn\Ap{Ap}
	\opn\mult{mult}
	\opn\bight{bight}
	\opn\div{div}
	\opn\Div{Div}
	\opn\aff{aff}
	\opn\relint{relint} \opn\st{st}
	\opn\lk{lk} \opn\cn{cn} \opn\core{core} \opn\vol{vol}  \opn\inp{inp}
	\opn\nilpot{nilpot}
	\opn\link{link} \opn\star{star}\opn\lex{lex}\opn\set{set}
	\opn\width{wd}
	\opn\Fr{F}
	\opn\QF{QF}
	\opn\G{G}
	\opn\type{type}\opn\res{res}
	\opn\conv{conv}
	\opn\Int{Int}
	\opn\Deg{Deg}
	\opn\Sym{Sym}
	\opn\Con{Con}
	\opn\gr{gr}
	
	\def\pot#1#2{#1[\kern-0.28ex[#2]\kern-0.28ex]}

	\opn\dirlim{\underrightarrow{\lim}}
	\opn\inivlim{\underleftarrow{\lim}}
	\let\to=\rightarrow
	
	\def\Implies{\ifmmode\Longrightarrow \else
		\unskip${}\Longrightarrow{}$\ignorespaces\fi}
	\def\implies{\ifmmode\Rightarrow \else
		\unskip${}\Rightarrow{}$\ignorespaces\fi}
	\def\iff{\ifmmode\Longleftrightarrow \else
		\unskip${}\Longleftrightarrow{}$\ignorespaces\fi}

	\let\:=\colon
	\newtheorem{Theorem}{Theorem}[section]
	\newtheorem{Lemma}[Theorem]{Lemma}
	\newtheorem{Corollary}[Theorem]{Corollary}

	\theoremstyle{definition}
	\newtheorem{Remark}[Theorem]{Remark}

	\newtheorem{Conjecture}[Theorem]{Conjecture}

    \makeatletter
\@namedef{subjclassname@2020}{%
  \textup{2020} Mathematics Subject Classification}
\makeatother

\begin{document}

\title[Bounded powers of edge ideals]{Bounded powers of edge ideals: symmetric exchange binomials}

\author[T.~Hibi]{Takayuki Hibi}
\author[S.~A.~ Seyed Fakhari]{Seyed Amin Seyed Fakhari}

\address{(Takayuki Hibi) Department of Pure and Applied Mathematics, Graduate School of Information Science and Technology, Osaka University, Suita, Osaka 565--0871, Japan}
\email{hibi@math.sci.osaka-u.ac.jp}
\address{(Seyed Amin Seyed Fakhari) Departamento de Matem\'aticas, Universidad de los Andes, Bogot\'a, Colombia}
\email{s.seyedfakhari@uniandes.edu.co}



\begin{abstract}
It has been conjectured that the toric ideal of the base ring of a discrete polymatroid is generated by symmetric exchange binomials.  In the present paper, we give several classes of discrete polymatroids which yield toric ideals generated by symmetric exchange binomials.  Especially, we are interested in the discrete polymatroids arising from bounded powers of edge ideals of finite graphs.
\end{abstract}

\maketitle

\section{Introduction} \label{sec1}
Let $S=K[x_1, \ldots,x_n]$ denote the polynomial ring in $n$ variables over a field $K$ with $\deg x_i = 1$, for each $i=1, \ldots, n$, and let $\Bc$ be a finite set of monomials of $S$.  Following \cite[p.~240]{HH_discrete}, we say that $\Bc$ is {\em the set of bases of a discrete polymatroid} if (i) all $u \in \Bc$ have the same degree and (ii) if $u = x_1^{a_1} \cdots x_n^{a_n}$ and $v = x_1^{b_1} \cdots x_n^{b_n}$ belong to $\Bc$ with $a_\xi > b_\xi$, then there is $\rho$ with $a_{\rho} < b_{\rho}$ for which $x_\rho(u/x_\xi)$ belongs to $\Bc$.  In the present paper, for simplicity, we call $\Bc$ a {\em polymatroid} instead of the set of bases of a discrete polymatroid.  A {\em matroid} is a polymatroid consisting of squarefree monomials.  A monomial ideal $I$ is called a {\em polymatroidal ideal} if $I$ is generated by a polymatroid.  Every polymatroidal ideal has linear quotients (\cite[Theorem 12.6.2]{HHgtm260}).  

Let $\Bc = \{w_1, \ldots, w_s\}$ be a polymatroid.  The subring $K[\Bc]=K[w_1, \ldots, w_s]$ is called the {\em base ring} of $\Bc$.  It follows from \cite[Theorem 12.5.1]{HHgtm260} that $K[\Bc]$ is normal and Cohen--Macaulay.  Let $T = K[z_1, \ldots, z_s]$ denote the polynomial ring in $s$ variables over a field $K$ and define the surjective ring homomorphism $\pi_\Bc : T \to K[\Bc]$ by setting $\pi_\Bc(z_i) = w_i$ for $1 \leq i \leq s$.  The {\em toric ideal} of $\Bc$ is $J_\Bc=\Ker(\pi_\Bc)$, the kernel of $\pi_\Bc$.  It follows from \cite[Theorem 12.4.1]{HHgtm260} that $\Bc$ satisfies the symmetric exchange property.  In other words, if $w_i = x_1^{a_1} \cdots x_n^{a_n}$ and $w_j = x_1^{b_1} \cdots x_n^{b_n}$ belong to $\Bc$ with $a_\xi > b_\xi$, then there is $\rho$ with $a_{\rho} < b_{\rho}$ for which both $x_\rho(w_i/x_\xi)$ and $x_\xi(w_j/x_\rho)$ belong to $\Bc$.  Let $w_{i_0}=x_\rho(w_i/x_\xi)$ and $w_{j_0}=x_\xi(w_j/x_\rho)$.  Then $z_iz_j - z_{i_0}z_{j_0}$ belongs to $J_\Bc$.  We call $z_iz_j - z_{i_0}z_{j_0}$ a {\em symmetric exchange binomial} of $\Bc$.

It has been conjectured (\cite{HH_discrete} and \cite{White}) that the toric ideal of a polymatroid is generated by symmetric exchange binomials.  In the proof of \cite[Theorem 5.3 (a)]{HH_discrete}, the following result is essentially proved:

\begin{Theorem}[\cite{HH_discrete}]
\label{HH_5.3(a)}
Fix $d \in \ZZ_{>0}$.  Suppose that the toric ideal of an arbitrary matroid consisting of squarefree monomials of degree at most $d$ is generated by symmetric exchange binomials.  Then the toric ideal of an arbitrary polymatroid consisting of monomials of degree at most $d$ is generated by symmetric exchange binomials.   
\end{Theorem}

Since the toric ideal of a matroid consisting of squarefree monomials of degree at most $3$ is generated by symmetric exchange binomials \cite{Blum, Kashi, compressed}, it follows that 

\begin{Corollary}
\label{IntroCor}
The toric ideal of a polymatroid consisting of monomials of degree at most $3$ is generated by symmetric exchange binomials.
\end{Corollary}

It is shown \cite{Nick} that the toric ideal of a transversal polymatroid is generated by symmetric exchange binomials.  Furthermore, it is shown \cite{Sch} that the toric ideal of a lattice path polymatroid is generated by symmetric exchange binomials.    

Now, in the present paper, we give several classes of polymatroids whose toric ideals are generated by symmetric exchange binomials.  Especially we are interested in the polymatroids arising from bounded powers of edge ideals of finite graphs.

First of all, in section \ref{sec2}, fundamental materials on finite graphs required in the present paper are summarized.

In Section \ref{sec3}, we discuss powers of edge ideals which are polymatroidal ideals.  Let $G$ be a finite graph without loops, multiple edges and isolated vertices.  Let $V(G) = \{x_1, \ldots, x_n\}$ denote the vertex set of $G$ and $E(G)$ the set of edges.  Recall that the {\em edge ideal} of $G$ is the ideal $I(G)$ of $S$ generated by those monomials $x_ix_j$ with $\{x_i, x_j\} \in E(G)$.  Let $\Gc(I(G))$ denote the minimal set of monomial generators of $I(G)$.  It is known \cite{Blum, compressed} that $\Gc(I(G))$ is a matroid if and only if $G$ is a complete multipartite graph, and that if $G$ is a complete multipartite graph, then the toric ideal of $\Gc(I(G))$ is generated by symmetric exchange binomials.  In Theorem \ref{THM_edge}, it is shown that $\Gc(I(G)^q)$, which is the minimal set of monomial generators of $I(G)^q$, is a polymatroid for some $q \geq 1$ if and only if $G$ is a complete multipartite graph, and that if $G$ is a complete multipartite graph, then $\Gc(I(G)^q)$ is a polymatroid for all $q \geq 1$ and the toric ideal $J_{\Gc(I(G)^q)}$ is generated by symmetric exchange binomials.

Furthermore, in Sections \ref{sec4} and \ref{sec5}, toric ideals arising from bounded powers of edge ideals studied in \cite{HSF1, HSF2, HSF3} are discussed.  Recall that what a bounded power of an edge ideal is.  Let $G$ be a finite graph on $V(G) = \{x_1, \ldots, x_n\}$ and $\mathfrak{c}=(c_1,\ldots,c_n)\in\ZZ_{>0}^n$.  Let $(I(G)^q)_\mathfrak{c}$ denote the monomial ideal generated by those $x_1^{a_1}\cdots x_n^{a_n} \in I(G)^q$ with each $a_i \leq c_i$.  Let $\delta_{\mathfrak{c}}(I(G))$ denote the biggest $q$ for which $(I(G)^q)_\mathfrak{c} \neq (0)$.  Then $(I(G)^{\delta_{\mathfrak{c}}(I(G))})_\mathfrak{c}$ is a polymatroidal ideal (\cite[Theorem 4.3]{HSF1}).  Let $\Bc(G,\mathfrak{c})$ denote the polymatroid which is the minimal set of monomial generators of $(I(G)^{\delta_{\mathfrak{c}}(I(G))})_\mathfrak{c}$ and $J_{\Bc(G,\mathfrak{c})}$ the toric ideal of $\Bc(G,\mathfrak{c})$.  A goal of our project is to prove the following  

\begin{Conjecture}
\label{IntroConj}
Let $G$ be a finite graph on $n$ vertices and $\mathfrak{c}\in\ZZ_{>0}^n$.  Then the toric ideal $J_{\Bc(G,\mathfrak{c})}$ of the polymatroid $\Bc(G,\mathfrak{c})$ is generated by symmetric exchange binomials.
\end{Conjecture}

First, in Section \ref{sec4}, based on Nicklasson \cite{Nick}, the product of polymatroids satisfying the strong exchange property is studied.  Note that the product of polymatroids is again a polymatroid.  We say that a polymatroid $\Bc$ enjoys the {\em strong exchange property} (\cite[Definition 2.5]{HH_discrete}) if, for all $u = x_1^{a_1} \cdots x_n^{a_n}$ and $v = x_1^{b_1} \cdots x_n^{b_n}$ belonging to $\Bc$ and for all $\xi$ and $\rho$ with $a_\xi > b_\xi$ and $a_{\rho} < b_{\rho}$, one has $x_\rho(w_i/x_\xi) \in \Bc$.  It follows from \cite[Theorem 5.3 (b)]{HH_discrete} that  the toric ideal of a polymatroid satisfying the strong exchange property is generated by symmetric exchange binomials.  Now, in \cite[Theorem 2.5]{Nick}, it is shown that if $\Bc_1, \ldots, \Bc_s$ are polymatroids satisfying the strong exchange property, then the toric ideal of the product 
\[
\Bc_1 \cdots \Bc_s = \left\{ \prod_{i=1}^{s} u_i : u_i \in \Bc_i \right\} 
\]
is generated by symmetric exchange binomials.  In our previous papers \cite{HSF2, HSF3}, we show that a finite graphs on $n$ vertices which is obtained from a complete multipartite graph by deleting one of its matchings enjoys the strong exchange property for all $\mathfrak{c}\in\ZZ_{>0}^n$ and we classify cycles, trees and unicyclic graphs on $n$ vertices which enjoy the strong exchange property for all $\mathfrak{c}\in\ZZ_{>0}^n$.  Now, combining \cite[Theorem 2.5]{Nick} with \cite{HSF2, HSF3} immediately yields a rich class of polymatroids whose toric ideals are generated by symmetric exchange binomials (Theorem \ref{SEP}).    

On the other hand, in Section \ref{sec5}, the first result is Theorem \ref{del2}, which says that if $\delta_{\mathfrak{c}}(I(G)) = 2$, then the toric ideal $J_{\Bc(G,{\mathfrak{c}})}$ of $\Bc(G,{\mathfrak{c}})$ is generated by symmetric exchange binomials.  Theorem \ref{del2} provides a class of polymatroids consisting of monomials of degree $4$ whose toric ideals are generated by symmetric exchange binomials.  The second result is Theorem \ref{path} which says that the toric ideal $J_{\Bc(P_n,\mathfrak{c})}$ of the polymatroid $\Bc(P_n,\mathfrak{c})$, where $P_n$ is the path of length $n-1$ on $n$ vertices, is generated by symmetric exchange binomials.

\section{Preliminaries} \label{sec2}

We summarize notations and terminologies on finite graphs. Let $G$ be a finite graph with no loop, no multiple edge and no isolated vertex on the vertex set $V(G)=\{x_1, \ldots, x_n\}$ and $E(G)$ the set of edges of $G$.
We say that $x_i \in V(G)$ is {\em adjacent} to (or is a {\em neighbor} of) $x_j \in V(G)$ if $\{x_i,x_j\} \in E(G)$.  If $x \in V(G)$, then $G - x$ is the finite graph on $V(G)\setminus \{x\}$ with $E(G - x) = \{e \in E(G) : x \not\in e \}$. A complete subgraph of $G$ is a {\em clique} of $G$.  A subset $A \subset V(G)$ is called {\em independent} if $\{x_i, x_j\} \not\in E(G)$ for all $x_i, x_j \in A$ with $x_i \neq x_j$. 

A {\em leaf} is a vertex with exactly one neighbor.  A {\em tree} is a finite connected graph with no cycle.  A {\em forest} is a finite graph with no cycle.  A {\em triangle} is the cycle of length $3$.  A {\em unicyclic} graph is a finite graph with exactly one cycle.  A {\em chordal} graph is a finite graph which has no induced cycle of length at least four.

A {\em matching} of $G$ is a subset $M \subseteq E(G)$ for which $e \cap e' = \emptyset$ for $e, e' \in M$ with $e \neq e'$.  An {\em induced} matching of $G$ is a matching $M$ of $G$ for which if $e, e' \in M$ with $e \neq e'$, then there is no $f \in E(G)$ with $f \cap e \neq \emptyset$ and $f \cap e' \neq \emptyset$.  The biggest cardinality of matchings (resp. induced matchings) of $G$ is called the {\em matching number} (resp. {\em induced matching number}) of $G$. We say that $G$ is a {\em Cameron--Walker} graph if the matching number and the induced matching number of $G$ are equal. If $M$ is a matching of $G$, then we define $G-M$ to be the finite graph obtained from $G$ by removing all edges belonging to $M$.  

Let $m \geq 2, n_1 \geq 1, \ldots, n_m \geq 1$ be integers and  
\[
V_i = \{x_{\sum_{j=1}^{i-1} n_j+1}, \ldots, x_{\sum_{j=1}^{i} n_j}\}, \, \, \, \, \, \, \, \, \, \, 1 \leq i \leq m.
\]
The finite graph $K_{n_1, \ldots, n_m}$ on $V(K_{n_1, \ldots, n_m}) = V_1 \sqcup \cdots \sqcup V_m$ with 
\[
E(K_{n_1, \ldots, n_m}) = \big\{ \{ x_k, x_\ell \} : x_k \in V_i, \, x_\ell \in V_{j}, \, 1 \leq i < j \leq m\big\}.
\]
is called the {\em complete multipartite graph} \cite[p.~394]{compressed} of type $(n_1, \ldots, n_m)$. For any integer $s\geq 1$, the complete multipartite graph of type $(1,s)$ is called a {\em star} graph.

\section{Polymatroids arising from powers of edge ideals} \label{sec3}

Let $G$ be a finite graph on $V(G) = \{x_1, \ldots, x_n\}$ allowing no loops, no multiple edges and no isolated vertices, and $E(G)$ the set of edges.  Let $I(G)$ denote the edge ideal of $G$ and $\Gc(I(G)) = \{x_ix_j : \{x_i, x_j\} \in E(G)\}$ the minimal set of monomial generators of $I(G)$.  It is known \cite{Blum, compressed} that $\Gc(I(G))$ is a matroid if and only if $G$ is a complete multipartite graph, and that if $G$ is a complete multipartite graph, then the toric ideal of $\Gc(I(G))$ is generated by symmetric exchange binomials.  

\begin{Theorem}
\label{THM_edge} 
Let $\Gc(I(G)^q)$ denote the minimal set of monomial generators of $I(G)^q$.  The following conditions are equivalent:
\begin{itemize}
    \item $\Gc(I(G)^q)$ is a polymatroid for some $q \geq 1$;
    \item $\Gc(I(G)^q)$ is a polymatroid for all $q \geq 1$;
    \item $G$ is a complete multipartite graph.
\end{itemize}
Furthermore, if $G$ is a complete multipartite graph, then for all $q \geq 1$, the toric ideal $J_{\Gc(I(G)^q)}$ of $\Gc(I(G)^q)$ is generated by symmetric exchange binomials.
\end{Theorem}

\begin{proof}  Let $G$ be a complete multipartite graph.  Then $\Gc(I(G))$ is a polymatroid and its toric ideal is generated by symmetric exchange binomials (\cite{Blum, compressed}). Thus, the product $\Gc(I(G)^q)$ is again a polymatroid for all $q \geq 1$.  Furthermore, $J_{\Gc(I(G)^q)}$ is generated by symmetric exchange binomials for all $q \geq 1$ (\cite[Theorem 3.8]{Shi}). 

Now, suppose that $\Gc(I(G)^q)$ is a polymatroid for some $q \geq 1$.  Our mission is to show that $G$ is a complete multipartite graph.  Suppose that $G$ is not a complete multipartite graph.  Then there are $4$ distinct vertices $x,y,z,w \in V(G)$ for which 
\[
\{x,y\}\not\in E(G), \, \, \, \{x,z\}\in E(G),  \, \, \, \{y,w\}\in E(G),  \, \, \, \{y,z\}\not\in E(G).  
\]
Since $(xz)^q, (yw)^q \in \Gc(I(G)^q)$, one has
either $(xz)^{q-1}(xy), (yw)^{q-1}(zw) \in \Gc(I(G)^q)$, or $(xz)^{q-1}(xw), (yw)^{q-1}(yz) \in \Gc(I(G)^q)$, which is impossible. 
\end{proof}

\begin{Remark}
\label{loop}
Let $\Bc$ be a polymatroid consisting of monomials of degree $2$ and $\Bc'$ the set of squarefree monomials belonging to $\Bc$.  Since $\Bc'$ is a matroid (Lemma \ref{squarefree}),  it follows that $\Bc' = \Gc(I(G))$, where $G$ is a complete multipartite graph.  If $x^2 \in \Bc$, then $x \in V(G)$.  If $x^2 \in \Bc$ and $yz \in \Bc$, then $xy, xz \in \Bc$.  In other words, $xy \in \Bc$ for all $y \in V(G)$.  As a result, in the vertex decomposition $V(G) = V_1 \sqcup \cdots \sqcup V_m$ of the complete multipartite graph $G$, if $x \in V_\ell$ and $x^2 \in \Bc$, then $V_\ell = \{x\}$.  
\end{Remark}

\begin{Lemma}
\label{squarefree}
Let $\Bc$ be a polymatroid and $\Bc'$ the set consisting of all squarefree monomials belonging to $\Bc$.  Then $\Bc'$ is a matroid. 
\end{Lemma}

\begin{proof}
Let $\Bc'=\{w_1, \ldots, w_m\}$. Assume that ${\rm deg}_{x_{\xi}}(w_i)>{\rm deg}_{x_{\xi}}(w_j)$. Since $\Bc$ is a polymatroid, there is $x_{\rho}$ with ${\rm deg}_{x_{\rho}}(w_i)<{\rm deg}_{x_{\rho}}(w_j)$ and  $x_{\rho}w_i/x_{\xi}\in \Bc$. Since $${\rm deg}_{x_{\rho}}(w_i)<{\rm deg}_{x_{\rho}}(w_j)\leq 1,$$ we deduce that $x_{\rho}$ does not divide $w_i$. Hence $x_{\rho}w_i/x_{\xi}$ is squarefree and, consequently, belongs to $\Bc'$, as desired.  
\end{proof}

\section{Product of polymatroids enjoying the strong exchange property} \label{sec4}

Recall what \cite[Theorem 2.5]{Nick} says.  Let $\Bc_1, \ldots, \Bc_s$ be polymatroids and set
\[
\prod_{i=1}^{d} \Bc_i := \left\{\prod_{i=1}^{s} u_i : u_i \in \Bc_i \right\}. 
\]
It turns out that $\prod_{i=1}^{s} \Bc_i$ is a polymatroid, which is called the {\em product} of $\Bc_1,\ldots, \Bc_s$.  Now, in \cite[Theorem 2.5]{Nick}, it is shown that 

\begin{Theorem}[\cite{Nick}]
\label{N}
Suppose that each of the polymatroid $\Bc_1, \ldots, \Bc_s$ enjoys the strong exchange property.  Then the the toric ideal of $\prod_{i=1}^{s} \Bc_i$ is generated by symmetric exchange binomials.  
\end{Theorem}

\begin{Corollary} 
\label{ZERO}
Suppose that the toric ideal of each of the polymatroid $\Bc_1, \ldots, \Bc_s$ is equal to $(0)$.  Then the toric ideal of $\prod_{i=1}^{s} \Bc_i$ is generated by symmetric exchange binomials.  Thus in particular, the toric ideal of a {\em trnasversal polymatroid} \cite[p.~267]{HH_discrete} is generated by symmetric exchange binomials.
\end{Corollary} 

Now, combining \cite[Theorem 2.5]{Nick} with \cite{HSF2, HSF3} immediately yields a rich class of polymatroids whose toric ideals are generated by symmetric exchange binomials.

\begin{Theorem} 
\label{SEP} 
Let $G$ be a finite graph on $n$ vertices and $\mathfrak{c}=(c_1,\ldots,c_n)\in\ZZ_{>0}^n$.  Suppose that each connected component of $G$ is of the form $K - M$, where $K$ is a complete multipartite graph and $M$ is a matching of $K$.  Then the toric ideal of $\Bc(G,\mathfrak{c})$ is generated by symmetric exchange binomials.
\end{Theorem}

\section{Polymatroids arising from bounded powers of edge ideals} \label{sec5}

We now turn to the study of toric ideals arising from bounded powers of edge ideals (\cite{HSF1, HSF2, HSF3}).  Let $G$ be a finite graph on $V(G) = \{x_1, \ldots, x_n\}$ without loops, multiple edges and isolated vertices.  Let $\mathfrak{c}=(c_1,\ldots,c_n)\in\ZZ_{>0}^n$.  Let $J_{\Bc(G,\mathfrak{c})}$ denote the toric ideal of the polymatroid $\Bc(G,\mathfrak{c})$.

If either $2\delta_{\mathfrak{c}}(I(G)) = c_1+\cdots+c_n$ or $2\delta_{\mathfrak{c}}(I(G)) = (c_1+\cdots+c_n)-1$, then the polymatroidal ring $K[\Bc(G,\mathfrak{c})]$ of $\Bc(G,\mathfrak{c})$ is the polynomial ring (\cite{HSF2, HSF3}).  In particular, $J_{\Bc(G,\mathfrak{c})} =(0)$.  On the other hand, since the toric ideal of a polymatroid consisting of monomials of degree at most $3$ is generated by symmetric exchange binomials (Corollary \ref{IntroCor}), it follows that 

\begin{Corollary}
\label{delta=c-2}
Let $G$ be a finite graph on $n$ vertices and $\mathfrak{c}\in\ZZ_{>0}^n$.  Suppose that either $2\delta_{\mathfrak{c}}(I(G)) = (c_1+\cdots+c_n)-2$ or $2\delta_{\mathfrak{c}}(I(G)) = (c_1+\cdots+c_n)-3$.  Then the toric ideal $J_{\Bc(G,\mathfrak{c})}$ of the polymatroid $\Bc(G,\mathfrak{c})$ is generated by the symmetric exchange binomials.
\end{Corollary}

\begin{proof}
Let $V(G)=\{x_1, \ldots, x_n\}$ and $\Bc(G,\mathfrak{c}) = \{u_1, \ldots, u_m\}$.
Furthermore, set $u:=x_1^{c_1}\cdots x_n^{c_n}$ and $w_i:=u/u_i$, for $1 \leq i \leq m$.  One has $K[\Bc(G,\mathfrak{c})] = K[u_1, \ldots, u_m] \cong K[w_1, \ldots, w_m]$.  We claim $\Bc'=\{w_1, \ldots, w_m\}$ is a polymatroid.  [{\em Proof.} Suppose ${\rm deg}_{x_{\xi}}(w_i)>{\rm deg}_{x_{\xi}}(w_j)$. Then ${\rm deg}_{x_{\xi}}(u_i)<{\rm deg}_{x_{\xi}}(u_j)$.  Since $\Bc$ is a polymatroid, the symmetric exchange theorem guarantees that there is $x_{\rho}$ with ${\rm deg}_{x_{\rho}}(u_i)>{\rm deg}_{x_{\rho}}(u_j)$ and $x_{\xi}u_i/x_{\rho}\in \Bc$.  In other words, ${\rm deg}_{x_{\rho}}(w_i)<{\rm deg}_{x_{\rho}}(w_j)$ and $x_{\rho}w_i/x_{\xi}\in \Bc'$.]  Since either all $w_i$'s are monomials of degree $2$ or all $w_i$'s are monomials of degree $3$, the desired result follows from Corollary \ref{IntroCor}. 
\end{proof}

We now come to the first highlight of the present section.  Theorem \ref{del2} provides a class of polymatroids consisting of monomials of degree $4$ whose toric ideals are generated by symmetric exchange binomials. 

\begin{Theorem} \label{del2}
Let $G$ be a finite graph on $n$ vertices and $\mathfrak{c}\in\ZZ_{>0}^n$.  Suppose that $\delta_{\mathfrak{c}}(I(G))=2$.  Then the toric ideal $J_{\Bc(G,\mathfrak{c})}$ of the polymatroid $\Bc(G,\mathfrak{c})$ is generated by symmetric exchange binomials.
\end{Theorem}

\begin{proof}
Let $\mathfrak{c}=(c_1, \ldots, c_n)\in \ZZ_{>0}^n$ and $\delta_{\mathfrak{c}}(I(G))=2$.  If there is a variable $x_i$ which divides all monomials belonging to $\Bc(G,\mathfrak{c})$, then $\Bc(G,\mathfrak{c}) = x_i\Bc'$, where $\Bc'$ is a polymatroid consisting of monomials of degree $3$. So, the assertion follows from Corollary \ref{IntroCor}.  If $G$ has a leaf $x$ whose unique neighbor is $y$, then all monomials belonging to $\Bc(G,\mathfrak{c})$ is divisible by $y$.  So, the assertion follows. 

Now, suppose that $G$ has no leaf.  In particular, $G$ is not a forest or a unicyclic graph.  Let $\ell$ denote the length of the longest induced cycle in $G$.  Since the matching number of $G$ is at most two, one has $\ell\leq 5$.  Let $\ell=5$.  Since the matching number of $G$ is at most two, one has $G=C_5$.  So, we are done (\cite[Theorem 3.5]{HSF3}). 

\medskip

{\bf (Case 1.)} Let $\ell=4$ and suppose that $G$ has an induced cycle of length $4$ on the vertices $x_1, x_2, x_3, x_4$ with edges $\{x_1, x_2\}, \{x_2, x_3\}, \{x_3, x_4\}, \{x_1, x_4\}$. 

{\bf Subcase 1.1.} Suppose that there is $1 \leq i \leq 4$ with $c_i \geq 2$. Without loss of generality, we may assume that $i=1$. Hence, $c_1 \geq 2$.  If $c_2\geq 2$, then $(x_1x_2)^2(x_3x_4)\in (I(G)^3)_{\mathfrak{c}}$, a contradiction. Thus,  $c_2=1$ and, similarly, $c_3=1$.  Recalling the first paragraph of the present proof, one may assume that there is $u \in \Bc(G,\mathfrak{c})$ which is not divisible by $x_2$.  If ${\rm deg}_{x_1}(u)\leq 1$, then $u(x_1x_2)\in (I(G)^3)_{\mathfrak{c}}$, a contradiction. Therefore, ${\rm deg}_{x_1}(u)\geq 2$ and then there is $x_5\in V(G)\setminus \{x_1, x_2, x_3, x_4\}$ with $\{x_1, x_5\}\in E(G)$.  This implies that $(x_1x_2)(x_1x_5)(x_3x_4)\in (I(G)^3)_{\mathfrak{c}}$, a contradiction.

{\bf Subcase 1.2.} Suppose that $c_1=c_2=c_3=c_4=1$.  Again, by the first paragraph of the present proof, one may assume that there is $u \in \Bc(G,\mathfrak{c})$ which is not divisible by $x_1$.  If ${\rm deg}_{x_2}(u)=0$, then $u(x_1x_2)\in (I(G)^3)_{\mathfrak{c}}$, a contradiction. Consequently, ${\rm deg}_{x_2}(u)=1$ and similarly ${\rm deg}_{x_4}(u)=1$. Thus, there is $x_5\in V(G)\setminus\{x_1, x_2, x_3, x_4\}$ which is adjacent to at least one of $x_2, x_4$.  Let, say, $\{x_2,x_5\}\in E(G)$.  Furthermore, since one may also assume that there is $v \in \Bc(G,\mathfrak{c})$ which is not divisible by $x_2$, there is $y\in V(G)\setminus\{x_1, x_2, x_3, x_4\}$ which is adjacent to at least one of $x_1,x_3$.  Let, say, $\{x_1,y\}\in E(G)$.  If $y\neq x_5$, then $(yx_1)(x_2x_5)(x_3x_4)\in (I(G)^3)_{\mathfrak{c}}$, a contradiction.  Thus $y=x_5$ and $\{x_1,x_5\}\in E(G)$. If $c_5\geq 2$, then $(x_1x_5)(x_2x_5)(x_3x_4)\in (I(G)^3)_{\mathfrak{c}}$, a contradiction. Therefore, $c_5=1$.  If $V(G)=\{x_1, \ldots, x_5\}$, then $c_1+\cdots +c_5=5$ and $2\delta_{\mathfrak{c}}(I(G))=(c_1+\cdots + c_5)-1$.  So, we are done. Let $V(G)\neq \{x_1, \ldots, x_5\}$.  It then follows easily that there is $x_6$ which is adjacent to one of $x_1, \ldots, x_5$ and the matching number of $G$ is at least $3$, a contradiction.

\medskip

{\bf (Case 2.)} Let $\ell=3$.  Then $G$ is a chordal graph. If $G$ is not a connected graph, then since $G$ has no leaf, we deduce that $G$ is a disjoint union of two triangles and $\mathfrak{c}=(1, \ldots, 1)\in \ZZ_{>0}^6$. Since $2\delta_{\mathfrak{c}}(I(G))=(c_1+\cdots +c_6)-2$, we are done by Corollary \ref{delta=c-2}.  

Suppose that $G$ is connected and that the induced matching number of $G$ is two.  Then $G$ is a Cameron--Walker graph. Since $G$ has no leaf, it follows from  \cite[p.~258]{CW} that $G$ is the union of two triangles sharing a vertex. Hence, $G = K_{2,2,1} - M$, where $M$ is a suitable matching of the complete multipartite graph $K_{2,2,1}$.  So, we are done (\cite[Theorem 2.1]{HSF3}).
 
Suppose that $G$ is connected and that the induced matching number of $G$ is one.  It is known  \cite[Theorem 2]{c} that $G$ has a clique $C$ for which $V(G)\setminus V(C)$ is an independent set of $G$.  Notice that one may choose $C$ to be a maximal clique of $G$.  Since $G$ has no leaf, it follows that every vertex in $V(G)\setminus V(C)$ is adjacent to at least two vertices of $C$.  

Let $m$ denote the number of vertices of $C$.  Since $\delta_{\mathfrak{c}}(I(G))=2$, one has $m \leq 5$. 
\begin{itemize}
\item
Let $m=5$.  One has $G=C$.  In fact, if $G \neq C$, then the matching number of $G$ must be at least three.  So, we are done (\cite[Theorem 2.1]{HSF3}). 
\item
Let $m=4$.  Since every vertex in $V(G)\setminus V(C)$ is adjacent to at least two vertices of $C$, we deduce that $V(G)\setminus V(C)$ is a singleton, as otherwise, the matching number of $G$ will be at least three.
If $V(G)\setminus V(C)$ is a singleton, then $G=K_{2,1,1,1}-M$, where $M$ is a suitable matching of the complete multipartite graph $K_{2,1,1,1}$.  So, we are done (\cite[Theorem 2.1]{HSF3}).
\item
Let $m=3$.  Since $C$ is a maximal clique of $G$, each vertex in $V(G)\setminus V(C)$ is adjacent to exactly two vertices of $C$. Let $V(C)=\{x_1, x_2, x_3\}$. Also, let $x_4\in V(G)\setminus V(C)$ with $\{x_1, x_4\}\in E(G)$, $\{x_2, x_4\}\in E(G)$ and $\{x_3, x_4\}\notin E(G)$.  If each $x_i\in V(G)\setminus V(C)$ satisfies $\{x_1, x_i\}\in E(G)$, $\{x_2, x_i\}\in E(G)$ and $\{x_3, x_i\}\notin E(G)$, then $G=K_{1,1,n-2}$ and we are done (\cite[Theorem 2.1]{HSF3}). So, suppose that there is $x_5\in V(G)\setminus V(C)$ with $\{x_3, x_5\}\in E(G)$ and exactly one of $\{x_1, x_5\}$ and $\{x_2, x_5\}$ is an edge of $G$.  If $V(G)=\{x_1, \ldots, x_5\}$, then $G=K_{2,2,1}-M$, where $M$ is a suitable matching of $K_{2,2,1}$. Therefore, we are done (\cite[Theorem 2.1]{HSF3}).  If $V(G)\neq \{x_1, \ldots, x_5\}$, then $G$ has vertex $x_6$ which is adjacent to exactly two of $x_1, x_2, x_3$ and, therefore, the matching number of $G$ is at least three, a contradiction.
\end{itemize}
Now, our proof of Theorem \ref{del2} is complete.
\end{proof} 

Let $P_{n}$ be the path of length $n-1$ on $V(P_{n})=\{x_1,\ldots, x_n\}$ with
\[
E(P_n) = \{ \{x_1,x_2\}, \{x_2,x_3\},\ldots, \{x_{n-1},x_n\} \}.
\] 
The second highlight of the present section is Theorem \ref{path} which says that the toric ideal $J_{\Bc(P_n,\mathfrak{c})}$ of the polymatroid $\Bc(P_n,\mathfrak{c})$ is generated by symmetric exchange binomials.  In our original definition of the bounded power $\delta_{\mathfrak{c}}(I(G))$ of the edge ideal $I(G)$ of a finite graph $G$ on $n$ vertices, we assume $\mathfrak{c} \in \ZZ_{>0}^n$.  Clearly our definition can be valid for $\mathfrak{c} \in \ZZ_{\geq 0}^n$.  

\begin{Lemma} \label{zeroone}
Let $\mathfrak{c}\in\{0,1\}^n$ with $n \geq 2$.  Then the toric ideal of $\Bc(P_n,\mathfrak{c})$ is generated by the symmetric exchange binomials.
\end{Lemma}

\begin{proof}
Let $\mathfrak{c}=(c_1, \ldots, c_n)$.  If each component of $\mathfrak{c}$ is one, then $2\delta_{\mathfrak{c}}(I(P_n))$ is either $c_1+\cdots +c_n$ or $c_1+\cdots +c_n-1$.  So, $\Bc(P_n,\mathfrak{c})$ is either $\{w\}$ or a subset of $\{w/x_1, \ldots, w/x_n\}$, where $w = x_1 \cdots x_n$.  Thus, in particular, the toric ideal of $\Bc(P_n,\mathfrak{c})$ is equal to $(0)$.

Suppose that $\mathfrak{c}$ has at least one component which is equal to zero.  Let $c_{i_1}, \ldots, c_{i_k}$ denote the zero components of $\mathfrak{c}$.  Let $G_1, \ldots, G_t$ denote the connected components of $P_n-\{x_{i_1}, \ldots x_{i_k}\}$.  Set  
\[
\mathfrak{c}_j:=(1, \ldots, 1)\in \ZZ_{> 0}^{|V(G_j)|}, \, \, \, \, \, 1 \leq j \leq t.  
\]
Since $c_{i_1}=\cdots =c_{i_k}=0$, it follows that 
$$\Bc(P_n,\mathfrak{c})=\Bc({G_1},\mathfrak{c_1})\cdots
\Bc({G_t},\mathfrak{c_t}).$$
Since the toric ideal of each $\Bc({G_i},\mathfrak{c_i})$ is equal to $(0)$, the desired result follows from Corollary \ref{ZERO}.
\end{proof}

\begin{Theorem}
\label{path}
Let $P_{n}$ be the path of length $n-1$ and $\mathfrak{c}=(c_1, \ldots, c_n)\in \ZZ_{\geq 0}^n$.  Then the toric ideal $J_{\Bc(P_n,\mathfrak{c})}$ of the polymatroid $\Bc(P_n,\mathfrak{c})$ is generated by symmetric exchange binomials.
\end{Theorem}

\begin{proof}
Let $V(P_n)=\{x_1, \ldots, x_n\}$ and $E(P_n)=\big\{\{x_i, x_{i+1}\} : 1\leq i\leq n-1\big\}$. Assume that 
\[
\Bc(P_n,\mathfrak{c})=\{v_1, \ldots, v_m\},
\]
which is the minimal set of monomial generators of $(I(P_n)^{\delta_{\mathfrak{c}}(I(P_n))})_{\mathfrak{c}}$. Let $T=K[z_1, \ldots z_m]$ be the polynomial ring in $m$ variables over the field $K$ and define the ring homomorphism $\pi : T \to K[v_1, \ldots v_m]$ by setting $\pi(z_i) = v_i$ for $1 \leq i \leq m$. Let $J:=J_{\Bc(P_n,\mathfrak{c})}=\Ker(\pi)$ denote the toric ideal of $\Bc(P_n,\mathfrak{c})$ and define $I$ to be the subideal of $J$ which is generated by the symmetric exchange binomials.  

Our mission is to show that $J=I$.  Set $e_i:=x_i,x_{i+1}$ for $1 \leq i \leq n-1$ and suppose that $v_j=e_1^{a_j(1)} \cdots e_{n-1}^{a_j(n-1)}$ for $1 \leq j \leq m$.  Given a monomial $u=z_1^{b_1}\cdots z_m^{b_m}\in T$, set 
\[
\tau(u):=\max\big\{|a_j(k)-a_{j'}(k)|: b_j\geq 1, b_{j'}\geq 1, 1 \leq k \leq n-1\big\}.
\]
Let $u-u'$ be a nonzero binomial in $J$.  Using induction on $\max\{\tau(u), \tau(u')\}$, we prove $u-u'\in I$.  Let $\tau(u)=\tau(u')=0$.  Then there are integers $1\leq j,j'\leq m$ and $\ell\geq 1$ with $u=z_j^{\ell}, u'=z_{j'}^{\ell}$.  Since $\pi(u)=\pi(u')$, one has $j=j'$.  Thus $u=u'$, which contradicts $u-u'\neq 0$. This implies that $\max\{\tau(u), \tau(u')\}\geq 1$.

\medskip

{\bf (Step 1.)} Suppose that $\max\{\tau(u), \tau(u')\}=1$.  Let $u=z_1^{b_1}\cdots z_m^{b_m}$ and $u'=z_1^{b'_1}\cdots z_m^{b'_m}$. For each $k=1, \ldots, n-1$, set $r_k:=\min\{a_j(k): b_j\geq 1\}$ and $r_k':=\min\{a_j(k): b_j'\geq 1\}$.

\medskip

{\bf Claim 1.} One has $r_k=r_k'$, for each $1 \leq k \leq n-1$.

\medskip
{\it Proof of Claim 1.}  Note that
$$\pi(u)=\prod_{i=1}^{n-1}e_i^{\sum_{j=1}^mb_ja_j(i)}.$$
Since $e_1$ (resp. $e_{n-1}$) is the unique edge of $P_n$ which is incident to $x_1$ (resp. $x_n$) and since, for each $2 \leq i \leq n-1$, the edges $e_{i-1}, e_i$ are the only edges of $P_n$ which are incident to $x_i$, it follows that 
\[
\begin{array}{rl}
& {\rm deg}_{x_1}(\pi(u))=\sum_{j=1}^mb_ja_j(1),\\
\\
& {\rm deg}_{x_i}(\pi(u))=\sum_{j=1}^mb_ja_j(i-1)+\sum_{j=1}^mb_ja_j(i), \, \, \, \, \, \, \, \, \, \, i=2, \ldots, n-1.
\end{array} \tag{1} \label{1}
\]
It follows from the above equalities that for $2 \leq i\leq n-1$, one has
\[
\begin{array}{rl}
\sum_{j=1}^mb_ja_j(i)=\sum_{t=0}^{i-1}(-1)^t{\rm deg}_{x_{i-t}}(\pi(u)).
\end{array} \tag{2} \label{2}
\]
Since $\tau(u)\leq 1$, one has $\max\{a_j(k): b_j\geq 1\}\in\{r_k, r_k+1\}$ for each $1 \leq k \leq n-1$.  Set $\ell:={\rm deg}(u)={\rm deg}(u')$. It then follows that
\[
\begin{array}{rl}
r_1=\Big\lfloor\frac{{\rm deg}_{x_1}(\pi(u))}{\ell}\Big\rfloor.
\end{array} \tag{3} \label{3}
\]
Furthermore, we deduce from (\ref{2}) that for each $2 \leq k \leq n-1$, 
\[
\begin{array}{rl}
r_k=\bigg\lfloor\frac{\sum_{j=1}^mb_ja_j(i)}{\ell}\bigg\rfloor=\bigg\lfloor\frac{\sum_{t=0}^{i-1}(-1)^t{\rm deg}_{x_{i-t}}(\pi(u))}{\ell}\bigg\rfloor.
\end{array} \tag{4} \label{4}
\]
Similarly,
\[
\begin{array}{rl}
r_1'=\Big\lfloor\frac{{\rm deg}_{x_1}(\pi(u'))}{\ell}\Big\rfloor,
\end{array} \tag{5} \label{5}
\]
and for each $2 \leq k \leq n-1$,
\[
\begin{array}{rl}
r_k'=\bigg\lfloor\frac{\sum_{t=0}^{i-1}(-1)^t{\rm deg}_{x_{i-t}}(\pi(u'))}{\ell}\bigg\rfloor.
\end{array} \tag{6} \label{6}
\]
Now, since $\pi(u)=\pi(u')$, Claim $1$ follows from equalities (\ref{3}), (\ref{4}), (\ref{5}) and (\ref{6}).

\medskip

Now, set
$$
\mathcal{A}:=\{v_j : a_j(k)\geq r_k, \ {\rm for \ each} \ k=1, \ldots, n-1\}.
$$
Let, say, $\mathcal{A}=\{v_1, \ldots, v_s\}$.  Also, let $\mathfrak{c'}$ denote the exponent vector of $\prod_{k=1}^{n-1}e_k^{r_k}$ and set $\mathfrak{c''}:=\mathfrak{c}-\mathfrak{c'}$. Then 
\[
\delta_{\mathfrak{c''}}(I(P_n))=\delta_{\mathfrak{c}}(I(P_n))-(r_1+\cdots +r_{n-1}). 
\]
Furthermore, if $w \in \Bc({P_n},{\mathfrak{c}''})$, then 
\[
w\prod_{k=1}^{n-1}e_k^{r_k} \in \Bc(P_n,\mathfrak{c}).
\]
Hence,
\[
\Bc({P_n},{\mathfrak{c}''})=\bigg\{\frac{v_j}{\prod_{k=1}^{n-1}e_k^{r_k}}: j=1, \ldots, s\bigg\}.
\]
Set $w_j:=\frac{v_j}{\prod_{k=1}^{n-1}e_k^{r_k}}$ for $1 \leq j \leq s$.  So,
\[
 \Bc({P_n},{\mathfrak{c}''})=\{w_1, \ldots, w_s\}.
\]
Let $T':=K[z_1, \ldots, z_s]$ be a subring of $T$ and $\pi' : T' \to K[w_1, \ldots w_s]$ the ring homomorphism defined by setting $\pi'(z_i)=w_i$ for $1 \leq i \leq s$.  Let $J'=\Ker(\pi')$ denote the toric ideal of $\Bc({P_n},{\mathfrak{c}''})$ and $I'\subseteq J'$ the subideal of $J'$ generated by the symmetric exchange binomials.  One has $u-u'\in J'$ and $I'\subseteq I$.  So, it is enough to prove that $u-u'\in I'$.  If no component of $\mathfrak{c''}$ is bigger than one, then we are done by Lemma \ref{zeroone}. So, assume that at least one component of $\mathfrak{c''}$ is at least two.  Let $\mathfrak{c''}=(c_1'', \ldots, c_n'')$ and $c_{i_1}'', \ldots, c_{i_h}''$, $i_1<\ldots < i_h$, the components of $\mathfrak{c''}$ which are bigger than one.  Let $G_1$ (resp. $G_{h+1}$) denote the induced subgraph of $P_n$ on $\{x_1, \ldots, x_{i_1}\}$ (resp. $\{x_{i_h}, \ldots, x_n\}$) and $G_q$ the induces subgraph of $P_n$ on $\{x_{i_{q-1}}, \ldots, x_{i_q}\}$ for $2 \leq q \leq h$.  Set 
\begin{align*}
& \mathfrak{c}_1'':=(c_1'', \ldots, c_{i_1-1}'', 1)\in \ZZ_{\geq 0}^{i_1},\\
& \mathfrak{c}_q'':=(1, c_{i_{q-1}+1}'',\ldots, c_{i_q-1}'', 1)\in \ZZ_{\geq 0}^{i_q - i_{q-1}+1}, \, \, \, \, \, q=2, \ldots h,\\
& \mathfrak{c}_{h+1}'':=(1, c_{i_h+1}'',\ldots, c_n'')\in \ZZ_{\geq 0}^{n-i_h+1}.
\end{align*}
Since $\tau(u), \tau(u')\leq 1$, we conclude from the definition of $r_k$ and Claim 1 that if $b_j\geq 1$ or $b_j'\geq 1$, then $w_j$ can be written as the product of distinct edges of $P_n$. This implies that for such $j$ one has
\[
\begin{array}{rl}
w_j\in \prod_{i=1}^{h+1} \Bc({G_i},{\mathfrak{c}_i''})
\subset \Bc(P_n,\mathfrak{c''}).
\end{array} \tag{7} \label{7}
\]
Since all components of each $\mathfrak{c}_i''$ are at most one, as was discussed in the proof of Lemma \ref{zeroone}, each polymatroid $\Bc(G_i,\mathfrak{c}_i'')$ is a product of polymatroids whose toric ideal  is $(0)$.  Hence, by  Corollary \ref{ZERO}, the toric ideal of $\prod_{i=1}^{h+1} \Bc(G_i,\mathfrak{c}_i'')$ is generated by symmetric exchange binomials.  Now, we know from the first containment of (\ref{7}) that $u-u'$ is contained in the toric ideal of $\prod_{i=1}^{h+1} \Bc({G_i},{\mathfrak{c}_i''})$.  Thus $u-u'\in I'$, as desired.

\medskip

{\bf (Step 2.)} Suppose that $\max\{\tau(u), \tau(u')\}\geq 2$.

\medskip

{\bf Claim 2.}  Given a monomial $w\in T$ with $\tau(w)\geq 2$, there is a monomial $w'\in T$ with $\tau(w') < \tau(w)$ and with $w-w'\in I$.

\medskip

{\it Proof of Claim 2.}  Let $w=z_1^{d_1}\cdots z_m^{d_m}$.  Since $\tau(w)\geq 2$, there are integers $j\neq j'$ with $d_j,d_{j'}\geq 1$ and an integer $1\leq k\leq n-1$ for which $a_j(k)-a_{j'}(k)=\tau(w)\geq 2$.  Set
$$\mathcal{B}_w:=\big\{(j,j',k) : 1\leq k \leq n-1, \, d_j, d_{j'}\geq 1, \, a_j(k)-a_{j'}(k)=\tau(w)\big\}.$$
First we show $|\mathcal{B}_w|\geq 2$.  Suppose that $|\mathcal{B}_w|=1$ and $\mathcal{B}_w=\{(j,j',k)\}$.  So, $d_j,d_{j'}\geq 1$ and $a_j(k)-a_{j'}(k)=\tau(w)$.  Let, say, $j=1$ and $j'=2$.  Thus, $d_1, d_2\geq 1$ and $a_1(k)-a_2(k)=\tau(w)\geq 2$.  We prove the $k \neq 1$ and $k \neq n-1$.      

Let $k=1$.  Since $e_1=\{x_1,x_2\}$, we deduce that
$$
{\rm deg}_{x_1}(v_2)=a_2(1)\leq a_1(1)-2={\rm deg}_{x_1}(v_1)-2\leq c_1-2.
$$
This yields that ${\rm deg}_{x_2}(v_2)=c_2$, as otherwise one has $(x_1x_2)v_2\in (I(P_n)^{\delta_{\mathfrak{c}}(I(P_n))+1})_{\mathfrak{c}}$, a contradiction.  Since ${\rm deg}_{x_2}(v_1)\leq c_2$ and since $a_1(1)=a_2(1)+\tau(w)$, one has $a_2(2)-a_1(2)=\tau(w)$. In other words, $(2,1,2)\in \mathcal{B}_w$, which means that $|\mathcal{B}_w|\geq 2$, a contradiction. Therefore, $k\neq 1$. Similarly, $k\neq n-1$, as desired.

Now, assume that $2\leq k\leq n-2$. Recall that $e_k=x_kx_{k+1}$.  A similar argument as above says that either ${\rm deg}_{x_k}(v_2)=c_k$ or ${\rm deg}_{x_{k+1}}(v_2)=c_{k+1}$.  Let ${\rm deg}_{x_k}(v_2)=c_k$.  (A similar argument is valid for ${\rm deg}_{x_{k+1}}(v_2)=c_{k+1}$.)  Since ${\rm deg}_{x_k}(v_1)\leq c_k$ and $a_1(k)=a_2(k)+\tau(w)$, one has $a_2(k-1)-a_1(k-1)=\tau(w)$. Thus, $(2,1,k-1)\in \mathcal{B}_w$ and $|\mathcal{B}_w|\geq 2$, a contradiction.  Therefore, $|\mathcal{B}_w|\geq 2$.

We prove Claim 2 by induction on $|\mathcal{B}_w|$. Let $|\mathcal{B}_w|=2$. By the preceding paragraph, we may assume that there are two consecutive integers $1\leq k,k+1\leq n-1$ and two integers $1\leq j,j'\leq m$ with $d_j,d_{j'}\geq 1$ such that $a_j(k)-a_{j'}(k)=\tau(w)$ and $a_{j'}(k+1)-a_j(k+1)=\tau(w)$.  Equivalently, $\mathcal{B}_w=\{(j,j'k), (j',j,k+1)\}$. Suppose that ${\rm deg}_{x_k}(v_{j'})=c_k$.  Since $a_j(k)-a_{j'}(k)=\tau(w)$, one has $k\geq 2$ and $a_{j'}(k-1)-a_j(k-1)=\tau(w)$.  This shows that $(j',j,k-1)\in \mathcal{B}_w$, a contradiction.  Consequently, ${\rm deg}_{x_k}(v_{j'})< c_k$. Similarly, ${\rm deg}_{x_{k+2}}(v_j)< c_{k+2}$. Thus,
$$e_kv_{j'}/e_{k+1}, e_{k+1}v_j/e_k\in \Bc(P_n,\mathfrak{c}).$$
In other words, there are integers $1\leq \ell, \ell'\leq m$ with
$$v_{\ell}=e_kv_{j'}/e_{k+1}, \, \, \, \, \, v_{\ell'}=e_{k+1}v_j/e_k.$$
Then $z_jz_{j'}-z_{\ell}z_{\ell'}$ is a symmetric exchange binomial.  Let, say, $d_j\leq d_{j'}$.  Then modulo $z_jz_{j'}-z_{\ell}z_{\ell'}$ the monomial $w$ is equal to $w(z_{\ell}z_{\ell'})^{d_j}/(z_jz_{j'})^{d_j}$.  Furthermore, since $\mathcal{B}_w=\{(j,j'k), (j',j,k+1)\}$, one has
$$\tau(w(z_{\ell}z_{\ell'})^{d_j}/(z_jz_{j'})^{d_j})<\tau (w).$$
So, set $w':=w(z_{\ell}z_{\ell'})^{d_j}/(z_jz_{j'})^{d_j}$, and the claim follows in the case $|\mathcal{B}_w|=2$.

Now, suppose that $|\mathcal{B}_w|\geq 3$. Assume that $(j,j',k)\in \mathcal{B}_w$. Therefore, $d_j,d_{j'}\geq 1$ and $a_j(k)-a_{j'}(k)=\tau(w)$. Let $k'$ be the smallest with $|a_j(k')-a_{j'}(k')|=\tau(w)$. Without loss of generality, we may assume that $a_j(k')-a_{j'}(k')=\tau(w)$. By a similar argument as in the case $|\mathcal{B}_w|=1$, we have $a_{j'}(k'+1)-a_j(k'+1)=\tau(w)$. Let $t\geq 1$ be the largest odd integer such that

$\bullet$ $a_j(k'+i)-a_{j'}(k'+i)=\tau(w)$, for each even integer $i$ with $0\leq i\leq t$; and 

$\bullet$ $a_{j'}(k'+i)-a_j(k'+i)=\tau(w)$, for each odd integer $i$ with $0\leq i\leq t$.

Assume that $a_j(k'+t+1)-a_{j'}(k'+t+1)=\tau(w)$. It follows from the choice of $t$ that$$a_{j'}(k'+t+2)-a_j(k'+t+2)<\tau(w).$$This yields that $${\rm deg}_{x_{k'+t+2}}(v_{j'})< {\rm deg}_{x_{k'+t+2}}(v_{j})\leq c_{k'+t+2}.$$It also follows from the choice of $k'$ that either $k'=1$ or $k'\geq 2$ and $a_{j'}(k'-1)-a_j(k'-1)<\tau(w)$. Therefore, $${\rm deg}_{x_{k'}}(v_{j'})< {\rm deg}_{x_{k'}}(v_{j})\leq c_{k'}.$$Consequently,$$(x_{k'}x_{k'+t+2})v_{j'}=(x_{k'}x_{k'+1})\cdots (x_{k'+t+1}x_{k'+t+2})\frac{v_{j'}}{e_{k'+1}e_{k'+3}\cdots e_{k'+t}}$$ belongs to $(I(P_n)^{\delta_{\mathfrak{c}}(I(P_n))+1})_{\mathfrak{c}}$, a contradiction. This contradiction shows that$$a_j(k'+t+1)-a_{j'}(k'+t+1)< \tau(w).$$This inequality and the equality $a_{j'}(k'+t)-a_j(k'+t)=\tau(w)$ imply that$${\rm deg}_{x_{k'+t+1}}(v_j)<{\rm deg}_{x_{k'+t+1}}(v_{j'})\leq c_{k'+t+1}.$$Similarly, one deduces from the choice of $k'$ and the equality $a_j(k')-a_{j'}(k')=\tau(w)$ that ${\rm deg}_{x_{k'}}(v_{j'})<c_{k'}$. Thus,
$$\frac{(e_{k'}e_{k'+2}\cdots e_{k'+t-1})v_{j'}}{(e_{k'+1}e_{k'+3}\cdots e_{k'+t})}, \frac{(e_{k'+1}e_{k'+3}\cdots e_{k'+t})v_j}{(e_{k'}e_{k'+2}\cdots e_{k'+t-1})}\in \Bc(P_n,\mathfrak{c}).$$
In other words, there are integers $1\leq p, p'\leq m$ with
$$v_p=\frac{(e_{k'}e_{k'+2}\cdots e_{k'+t-1})v_{j'}}{(e_{k'+1}e_{k'+3}\cdots e_{k'+t})}, \, \, \, \, \, v_{p'}=\frac{(e_{k'+1}e_{k'+3}\cdots e_{k'+t})v_j}{(e_{k'}e_{k'+2}\cdots e_{k'+t-1})}.$$
Then $z_jz_{j'}-z_pz_{p'}$ is a symmetric exchange binomial. Let, say, $d_j\leq d_{j'}$. Then modulo $z_jz_{j'}-z_pz_{p'}$ the monomial $w$ is equal to $w(z_pz_{p'})^{d_j}/(z_jz_{j'})^{d_j}$. Set $w'':=w(z_pz_{p'})^{d_j}/(z_jz_{j'})^{d_j}$. So, $w-w''\in I$. Since $(j,j',k'), (j',j, k'+t)\in \mathcal{B}_w$, we have either (i) $\tau(w'')<\tau(w)$ or (ii) $\tau(w'')=\tau(w)$ and $|\mathcal{B}_{w''}|<|\mathcal{B}_w|$. In (i) set $w':=w''$. In (ii) the existence of $w'$ is guaranteed by the induction hypothesis.  This completes the proof of Claim 2. 

\medskip

We now turn to the proof of Step 2.  Let, say, $\max\{\tau(u), \tau(u')\}=\tau(u) \geq 2$. It follows from Claim 2 that there is a monomial $u_0\in T$ with $\tau(u_0)< \tau(u)$ and $u-u_0\in I$.  If $\tau(u')< \tau(u)$, then
$$\max\{\tau(u_0), \tau(u')\}< \max\{\tau(u), \tau(u')\}.$$
Therefore, the induction hypothesis implies that $u_0-u'\in I$, which implies that $u-u'\in I$.  Let $\tau(u')=\tau(u)$.  It follows from Claim 2 that there is a monomial $u_0'\in T$ with $\tau(u_0')< \tau(u')$ and $u'-u_0'\in I$. Hence,$$\max\{\tau(u_0), \tau(u_0')\}< \max\{\tau(u), \tau(u')\},$$
and the induction hypothesis implies that $u_0-u_0'\in I$. Consequently, $u-u'\in I$, as desired.  This completes the proof of Theorem \ref{path}.  
\end{proof}

Finally, we give another result which might be useful for our further research.

\begin{Theorem}
\label{deletion}
Let $G$ be a finite graph on $n$ vertices and $x$ a leaf of $G$.  Suppose that for each $\mathfrak{c}\in \ZZ_{>0}^n$, the toric ideal $J_{\Bc(G,\mathfrak{c})}$ of the polymatroid $\Bc(G,\mathfrak{c})$ is generated by symmetric exchange binomials.  Then for each $\mathfrak{c'}\in \ZZ_{>0}^{n-1}$, the toric ideal $J_{\Bc(G-x,\mathfrak{c'})}$ of the polymatroid $\Bc(G-x,\mathfrak{c'})$ is generated by symmetric exchange binomials. 
\end{Theorem}

\begin{proof}
Let $V(G)=\{x_1, \ldots, x_n\}$ and $x=x_n$ a leaf of $G$.  Let $x_{n-1}$ be the unique neighbor of $x$.  Set $H:=G-x_n$.  Fix $\mathfrak{c'}=(c, \ldots, c_{n-1)}\in \ZZ_{>0}^{n-1}$. Our goal is to show that the toric ideal $J_{\Bc(H,\mathfrak{c'})}$ is generated by symmetric exchange binomials.  Define  $\mathfrak{c}=(c_1, \ldots, c_n)\in \ZZ_{>0}^n$ by setting
$$c_i=
\left\{
	\begin{array}{ll}
		c_i'  &  1\leq i\leq n-2,\\
        c_{n-1}'+1  &  i=n-1,\\
		  1 & i=n.
	\end{array}
\right.$$ 
Set $\delta:=\delta_{\mathfrak{c}}(I(G))$ and $\delta':=\delta_{\mathfrak{c'}}(I(H))$.  Since $x_{n-1}$ is the unique neighbor of $x_n$, it follows that $\delta=\delta'+1$ and
$$(x_{n-1}x_n)(I(H)^{\delta'})_{\mathfrak{c'}}\subseteq (I(G)^{\delta})_{\mathfrak{c}}.$$
Let $\Bc(H, \mathfrak{c'})=\{v_1, \ldots, v_p\}$, which is the minimal set of monomial generators of $(I(H)^{\delta'})_{\mathfrak{c'}}$.  Set $u_i:=(x_{n-1}x_n)v_i$ for each $1 \leq i \leq p$.  We know from the above inclusion that each $u_i$ belongs to $\Bc(G, \mathfrak{c})$, which is the minimal set of monomial generators of $(I(G)^{\delta})_{\mathfrak{c}}$.  Let $\Bc(G, \mathfrak{c})=\{u_1, \ldots, u_p, u_{p+1}, \ldots, u_q\}$.

Let $T=K[z_1, \ldots z_q]$ denote the polynomial ring in $q$ variables over the field $K$ and define the surjective ring homomorphism $\pi : T \to K[u_1, \ldots u_q]$ by setting $\pi(z_i)=u_i$ for $1 \leq i \leq q$. Similarly, set $T'=K[z_1, \ldots, z_p]\subseteq T$ and define the surjective ring homomorphism $\pi' : T' \to K[v_1, \ldots, v_p]$ by setting $\pi'(z_i)=v_i$ for $1 \leq i \leq q$.  Then $$J_{\Bc(H,\mathfrak{c'})}=\Ker(\pi')\subseteq \Ker(\pi)=J_{\Bc(G,\mathfrak{c})}.$$Let $w-w'$ be a binomial in $J_{\Bc(H,\mathfrak{c'})}$. So, $w-w'\in J_{\Bc(G,\mathfrak{c})}$. It follows from the assumption that
\[
\begin{array}{rl}  
w-w'=\sum_{r=1}^kc_rw_r''(w_r-w_r'),
\end{array} \tag{8} \label{8}
\]
where for each $1 \leq r \leq k$, one has $c_r\in K$ and $w_r, w_r', w_r''$ are monomials in $T$ such that $w_r-w_r'$ is a symmetric exchange binomial of $\Bc(G,\mathfrak{c})$. Set $\ell:={\rm deg}(w)={\rm deg}(w')$. Since $J_{\Bc(G,\mathfrak{c})}$ is a homogeneous ideal, it may be assumed that each term of (\ref{8}) has degree $\ell$.  Now, $T$ is a $\ZZ^n$-graded ring, where the degree of each $z_i$ is the exponent vector of $u_i$.  Let ${\rm multdeg}(.)$ denote this multidegree.  We know that for each $z_i\in T'$, the last entry of ${\rm multdeg}(z_i)$ is one.  Furthermore, for each $z_i\in T\setminus T'$, the last entry of ${\rm multdeg}(z_i)$ is zero. As $w,w'\in T'$, it follows that the last entry of ${\rm multdeg}(w)={\rm multdeg}(w')$ is $\ell$. Since $J_{\Bc(G,\mathfrak{c})}$ is homogeneous with respect to this multigree, it may be assumed that in (\ref{8}) the last entry of the multidegree of each term is ${\ell}$. As the (ordinary) degree of each term of (\ref{8}) is $\ell$, we deduce that $w_r, w_r', w_r''$ belong to $T'$ for each $1 \leq r \leq k$. Consequently, $w-w'$ belongs to the ideal which is generated by symmetric exchange binomials of  $\Bc(H,\mathfrak{c'})$. This completes the proof of Theorem \ref{deletion}.
\end{proof}

\section*{Acknowledgments}
The second author is supported by a FAPA grant from Universidad de los Andes.

\section*{Statements and Declarations}
The authors have no Conflict of interest to declare that are relevant to the content of this article.

\section*{Data availability}
Data sharing does not apply to this article as no new data were
created or analyzed in this study.

\end{document}